\documentclass[11pt]{article}
\usepackage{mathrsfs}
\usepackage{amsmath}
\usepackage{amssymb}
\usepackage{graphicx}
\usepackage{epic}
\usepackage{color}

\usepackage{pst-poly}  
\usepackage{pst-plot}  
\usepackage{pst-poly}  
\usepackage{graphicx}

\renewcommand{\paragraph}{\roman{paragraph}}

\usepackage{bm}
\usepackage{hyperref}
\usepackage{tikz}
\usetikzlibrary{arrows,shapes,positioning}
\usetikzlibrary{decorations.markings}
\tikzstyle arrowstyle=[scale=1]
\tikzstyle directed=[postaction={decorate,decoration={markings, mark=at position .65 with {\arrow[arrowstyle]{stealth}}}}]
\tikzstyle reverse directed=[postaction={decorate,decoration={markings, mark=at position .65 with {\arrowreversed[arrowstyle]{stealth};}}}]

\newtheorem{theorem}{Theorem}[section]
\newtheorem{corollary}[theorem]{Corollary}
\newtheorem{definition}[theorem]{Definition}

\newtheorem{lemma}[theorem]{Lemma}

\newenvironment{proof}{\noindent {\bf Proof.}}{\rule{3mm}{3mm}\par\medskip}

\usepackage{amssymb}
\usepackage{graphics}
\usepackage{graphicx}
\usepackage{subfigure}
\usepackage{epsfig}
\usepackage{amsmath}
\usepackage{bm}
\usepackage{lineno}
\usepackage{url}
\usepackage{float}
\usepackage{color}

\begin{document}

\title{On mixed graphs whose Hermitian spectral radii are at most  $2$ \thanks{This work was supported by National Natural Science Foundation of China(11771016, 11871073)
}}

\author{Bo-Jun Yuan$^{a}$, Yi Wang$^{a}$
\thanks{Corresponding
author. E-mail address: wangy@ahu.edu.cn(Y. Wang), ybjmath@163.com(B.-J. Yuan), scgong@zafu.edu.cn(S.-C. Gong), qiaoyunahu@126.com(Y. Qiao)}
, Shi-Cai Gong$^{b}$, Yun Qiao$^{a}$ \\
{\small  \it a. School of Mathematical Sciences, Anhui University, Hefei, 230601, P. R. China}\\
{\small  \it b. School of Science, Zhejiang University of Science and Technology, Hangzhou, 310023, P. R. China}
}

\date{}
\maketitle

\begin{abstract}
A mixed graph is a graph with  undirected  and directed edges. Guo and Mohar in 2017 determined all mixed graphs whose  Hermitian spectral radii  are  less  than   $2$.
In this  paper, we give a sufficient condition which can make Hermitian spectral radius of a connected mixed graph strictly decreasing when an edge or a vertex is deleted, and characterize all  mixed graphs with  Hermitian spectral radii at most  $2$ and  with no cycle of length $4$  in their  underlying  graphs.
\end{abstract}



\section{Introduction}
Characterizing the structure of a graph by the eigenvalue  spectrum of an associated  matrix with  the graph
is a basic problem in spectral graph theory.
Since restrictions on the spectral radii of graphs   with respect to their adjacency  matrices    often force those to have   very special structures, it is always a hot topic to characterize the   graphs whose spectral radii
are  bounded above.
Smith \cite{Smith}    determined all  graphs whose
spectral radii  are  at most  $2$.
This  work stimulated the interest of the researchers. There are a lot of   results in the  literature concerning   the topic.
Brouwer and Neumaier \cite{bro} characterized   the graphs whose  spectral radii are contained in the interval $(2, \sqrt{2+\sqrt{5}}]$  and later, Woo and Neumaier \cite{woo}   described  the structure of graphs whose spectral radii are  bounded above by $\tfrac{3}{2}\sqrt{2}$.

Studying the same problem on digraphs   has received less attention.  Xu and Gong \cite{gong} investigated digraphs whose spectral radii  with respect to their  skew adjacency matrices  do not exceed $2$.
Guo and   Mohar \cite{mohar3}   determined    all mixed graphs   whose spectral radii  with respect to their   Hermitian  adjacency matrices  are less  than $2$.
Their work   shows  that $2$ is the smallest limit point  of the Hermitian spectral radii of connected mixed graphs.
In the present   paper,  we  characterize   all $C_4$-free mixed graphs whose Hermitian spectral radii do not exceed $2$.

A graph  containing undirected edges and directed edges is called a  {\sl mixed graph}.
Clearly, mixed graphs   are  natural  generalizations of both simple   graphs and digraphs.
Indeed, a mixed graph  $D$   can be obtained from a simple graph $G$  by orienting a subset of its edge set.
We call $G$ as   the {\sl underlying graph} of $D$  and  denote it  by $G(D)$.
Formally, a mixed graph $D$ is comprised of the vertex set $V(D)$, which is the same as the vertex set $V(G)$, and the edge set $E(D)$, which
consists of two parts: undirected edge set $E_0(D)$ and directed edge   set $E_1(D)$.
To  distinguish  undirected and  directed edges, we  denote an undirected edge between the  vertices  $u$ and $v$  by $\{u,v\}$ and a  directed edge from   $u$ to   $v$ by $(u, v)$.  If there is no danger of confusion,   we   write $uv$ instead of $\{u, v\}$ or $(u, v)$.

Let $D$ be a mixed graph of order $n$.
For a vertex $v$ of $D$,  we define the   set of  neighbors  of $v$ as $N(v)=\{u\in V(D) \, | \, uv\in E(G(D))\}.$ The
{\sl degree} of $v$ is defined as  $d(v)=|N(v)|$.
A mixed graph is said to be a  {\sl mixed tree}   (respectively,    {\sl unicyclic   mixed
graph})   if its underlying graph is a tree (respectively,   unicyclic graph).
A mixed subgraph $H$ of  $D$  is called  {\sl elementary}    if each connected component of $H$  is either a mixed  edge or a mixed   cycle.
A mixed graph $D$ is called {\sl $C_4$-free}  if $G(D)$ contains no  cycle of  length $4$   as a  subgraph.

The {\sl Hermitian adjacency matrix}  of
$D$ is defined as   $H(D)=[h_{uv}]$ with
$$h_{uv}= \left\{
\begin{array}{ll}
1 & {\rm if}\  \{u,v\}\  \in E_0(D);\\
\mathsf{i} & {\rm if}\  (u,v)\  \in E_1(D);\\
-\mathsf{i} & {\rm if}\  (v,u)\  \in E_1(D);\\
0 & {\rm otherwise,}
\end{array}
\right.$$
where $\mathsf{i}$ is the unit imaginary number.
Since $H(D)$ is Hermitian, the  eigenvalues of $H(D)$  are real and can be arranged as
$\lambda_1(D)\geqslant\cdots\geqslant\lambda_n(D)$.
The eigenvalues  and  spectrum  of $H(D)$ are called the {\sl  Hermitian eigenvalues}  and {\sl Hermitian spectrum}  of $D$, respectively.
The {\sl Hermitian spectral radius} of   $D$  is defined as
$\rho(D)=\max\{|\lambda_1(D)|,  \ldots, |\lambda_n(D)|\}$.
The characteristic polynomial of $H(D)$  is  denoted by $\mathnormal{\Phi}(D, \lambda)$ and is called  the  {\sl Hermitian  characteristic polynomial} of $D$.
These terminologies were introduced by Liu and Li \cite{li} in the study of graph energy and independently by Guo and Mohar \cite{mohar2}. In the paper \cite{li}, authors investigate the properties of
characteristic polynomials of mixed graphs and cospectral problems among mixed graphs. The latter paper contains an introduction to the properties of Hermitian spectrum, and discusses similarities and differences from
the case of undirected graph.  Recently, the Hermitian spectrum has been the subject of several publications.  For more details about the Hermitian spectrum, one can see the literature \cite{chen1, chen2, chenx, gre,hu, mohar1, wang} and references therein.

In this paper, we  deal with   Hermitian spectral radii of mixed graphs. A mixed graph $D$ is called
{\sl $C_4$-free} if $G(D)$ contains no $C_4$ as a  subgraph. We  characterize   all $C_4$-free mixed graphs whose Hermitian spectral radii are at most  $2$. The rest
of the paper is organized as follows:
In Section 2, we will introduce some notations and preliminary results on characteristic polynomials of mixed graphs.
In Section 3,  we will give a sufficient condition which can make Hermitian spectral radius of a connected mixed graph strictly decreasing when an edge or a vertex is deleted.
In Section 4, we will determine all $C_4$-free mixed  graphs whose Hermitian spectral radii do not exceed $2$.

\section{Notations and Preliminaries}

Let  $H(D)=[h_{ij}]$ be the Hermitian adjacency matrix of the mixed graph  $D$. The {\sl value} of a mixed walk $W: v_1, v_2, \ldots,  v_\ell$ is defined to be $h_{12}h_{23}\cdots h_{(\ell-1)\ell}$ and is  denoted by $h(W)$.
For a   closed  mixed walk $W$, we first fix  an arbitrary  direction for $W$ before calculating its value. One can verify that if  the value of a closed  mixed walk  is $\alpha$ in  a direction, then for the reversed direction its value is $\overline{\alpha}$,  the conjugate number of $\alpha$.
We say a mixed  cycle $C$ to be  {\sl real}   (respectively, {\sl imaginary})  if   $h(C)=\pm1$   (respectively, $\pm\mathsf{i}$).  It is clear that  a mixed   cycle  is  real    (respectively,   imaginary)   if  and only if  the       number of  its     directed edges is       even    (respectively,  odd).  Indeed, the value of a real cycle $C$ is independent of its chosen orientation.
Furthermore, a mixed  cycle $C$   is called {\sl positive} (respectively, {\sl negative}) if $h(C)=1$ (respectively, $-1$).   Clearly, a mixed   cycle    is   positive (respectively,  negative) if  and only if  the difference between the   number of   its  forward and  backward directed edges   with respect to  an arbitrary direction  is congruence  to $0$    (respectively, $2$) modulo  $4$.

We here  recall the following theorem  which can be considered as an  analogue of  Sachs' Coefficient Theorem \cite[Page 32]{cve}.

\begin{theorem}{\em \cite{mohar2, li}}\label{cha}
Let $D$ be a mixed graph of order $n$ with  the Hermitian  characteristic polynomial $\mathnormal{\Phi}(D, \lambda)=\sum_{i=0}^nc_i\lambda^{n-i}$. Denote by  $\mathcal{E}_i$   the set   of the  elementary subgraphs of $D$ of order $i$ whose all mixed cycles are real. Then for   $i=1, \ldots, n$,
$$c_i=\sum_{H\in\mathcal{E}_i}(-1)^{t(H)+s(H)}2^{r(H)},$$
where   $t(H)$, $s(H)$, and $r(H)$ are   respectively  the number of connected  components,   the number of negative mixed cycles,  and   the number of    mixed cycles in $H$.
\end{theorem}

The following   two corollaries  can be considered as immediate consequences of Theorem \ref{cha}.

\begin{corollary}{\em \cite{li}}\label{symmetric}
If a mixed graph $D$ contains no   real  mixed odd cycles, then the  Hermitian spectrum of $D$ is symmetric about   $0$.
\end{corollary}

\begin{corollary}{\em \cite{li}} \label{positive}
If all the mixed cycles in  a mixed graph $D$ are positive, then the Hermitian spectra  of $D$ and  $G(D)$ are the same. In particular,  $F$ and  $G(F)$ have the same  Hermitian spectrum for any mixed forest $F$.
\end{corollary}

In the following, we apply Theorem \ref{cha} to determine all mixed cycles with Hermitian spectral radii $2$.

\begin{corollary}\label{undercycle}
Let $D$ be a mixed cycle. Then $\rho(D)=2$ if and only if either $D$ is a positive  mixed cycle or  $D$ is a   negative   mixed  odd cycle.
\end{corollary}

\begin{proof}
Let $n=|V(D)|$ and $C=G(D)$. We know that  $\rho(C)=2$. In addition,   $\mathnormal{\Phi}(C, -2)=0$ if and only if $n$ is even. So, using  the Perron-Frobenius theorem,   $\rho(D)\leqslant\rho(C)=2$.  Define
\begin{equation*}
\begin{array}{lll}
t=\left\{\begin{array}{ll} 1 & \mbox{if $D$ is real};  \\ 0 & \mbox{otherwise}  \end{array}\right.  & \hspace{1cm} \text{ and } \hspace{1cm} &
s=\left\{\begin{array}{ll}   1 & \mbox{if $D$ is negative};   \\   0 & \mbox{otherwise}.  \end{array}\right.
\end{array}
\end{equation*}
By Theorem \ref{cha}, $\mathnormal{\Phi}(D, \lambda)-\mathnormal{\Phi}(C, \lambda)=(-1)^{s+1}2t+2$.
This means that $\mathnormal{\Phi}(D, 2)=0$ if and only if $t=1$ and $s=0$.
If $n$ is even, then the equality $\mathnormal{\Phi}(D, \lambda)-\mathnormal{\Phi}(C, \lambda)=(-1)^{s+1}2t+2$  shows  that $\mathnormal{\Phi}(D, -2)=0$ if and only if $t=1$ and $s=0$.
If $n$ is odd, then it follows from  Theorem \ref{cha}  that  $\mathnormal{\Phi}(D, \lambda)+\mathnormal{\Phi}(C, -\lambda)=(-1)^{s+1}2t-2$.
Hence, if  $n$ is odd, then  $\mathnormal{\Phi}(D, -2)=0$ if and only if $t=1$ and $s=1$.
\end{proof}

We prove  the following theorem as a consequence of Theorem \ref{cha}.

\begin{theorem}\label{de}
Let $D$ be a mixed graph and  $e=uv\in E(D)$. Let $\mathcal{C}_e$ be  the set of  all real  mixed cycles   in $D$  containing $e$. Then
$$\mathnormal{\Phi}(D, \lambda)=\mathnormal{\Phi}(D-e, \lambda)-\mathnormal{\Phi}(D-u-v, \lambda)-2\sum_{C\in\mathcal{C}_e}h(C)\mathnormal{\Phi}(D-C, \lambda).$$
\end{theorem}

\begin{proof}
Let $n=|V(D)|$  and  $\mathnormal{\Phi}(D, \lambda)=\sum_{\ell=0}^nc_\ell\lambda^{n-\ell}$   and fix $i\in\{1, \ldots, n\}$.  Let  $\mathcal{E}_i$ be the set  of all elementary subgraphs of $D$ of order $i$ whose all mixed cycles are real.  For any given edge $e$, $\mathcal{E}_i$ can be divided into the following subsets:

$$\begin{array}{llll}
\mathcal{E}_i^1=\{H\in\mathcal{E}_i \, | \,  e\notin E(H)\};\\
\mathcal{E}_i^2=\{H\in\mathcal{E}_i \, | \,  e \text{ is a  single edge of } H\};\\
\mathcal{E}_i^3=\{H\in\mathcal{E}_i \, | \,  e \text{ is contained in a  positive mixed cycle } P_H \text{ of } H\};\\
\mathcal{E}_i^4=\{H\in\mathcal{E}_i \, | \,  e \text{ is contained in a negative mixed cycle  } N_H \text{ of } H\}.\\
\end{array}$$
By Theorem \ref{cha},  we have
$$\sum_{H\in\mathcal{E}_i^1}(-1)^{t(H)+s(H)}2^{r(H)}=c_i(D-e),$$
\begin{align*}\sum_{H\in\mathcal{E}_i^2}(-1)^{t(H)+s(H)}2^{r(H)}&=-\sum_{H\in\mathcal{E}_i^2}(-1)^{t(H-u-v)+s(H-u-v)}2^{r(H-u-v)}\\&=-c_{i-2}(D-u-v),\end{align*}
\begin{align*}\sum_{H\in\mathcal{E}_i^3}(-1)^{t(H)+s(H)}2^{r(H)}&=-2\sum_{H\in\mathcal{E}_i^3}(-1)^{t(H-P_H)+s(H-P_H)}2^{r(H-P_H)}\\&=-2\sum_{C\in\mathcal{C}_e^+}c_{i-|V(C)|}(D-C),\end{align*}
and
\begin{align*}\sum_{H\in\mathcal{E}_i^4}(-1)^{t(H)+s(H)}2^{r(H)}&=2\sum_{H\in\mathcal{E}_i^4}(-1)^{t(H-N_H)+s(H-N_H)}2^{r(H-N_H)}\\&=2\sum_{C\in\mathcal{C}_e^-}c_{i-|V(C)|}(D-C),\end{align*}
where
$\mathcal{C}_e^+$ (respectively, $\mathcal{C}_e^-$)  is   the set   of  all  positive (respectively, negative)  mixed cycles  in $D$  containing $e$.
Now, it follows from Theorem \ref{cha} and   $\mathcal{E}_i=\mathcal{E}_i^1\cup\cdots\cup\mathcal{E}_i^4$  that
$$c_i(D)=c_i(D-e)-c_{i-2}(D-u-v)-2\sum_{C\in\mathcal{C}_e}h(C)c_{i-|V(C)|}(D-C),$$
which in turn implies that
\begin{align*}\sum_{i=0}^nc_i(D)\lambda^{n-i}&=\sum_{i=0}^nc_i(D-e)\lambda^{n-i}-\sum_{i=2}^nc_{i-2}(D-u-v)\lambda^{n-i}\\&
-2\sum_{C\in\mathcal{C}_e}h(C)\sum_{i=|V(C)|}^nc_{i-|V(C)|}(D-C)\lambda^{n-i}.\end{align*}
This means that    $\mathnormal{\Phi}(D, \lambda)=\mathnormal{\Phi}(D-e, \lambda)-\mathnormal{\Phi}(D-u-v, \lambda)-2\sum_{C\in\mathcal{C}_e}h(C)\mathnormal{\Phi}(D-C, \lambda)$.
\end{proof}

The next result  can be proved by applying Theorem \ref{de} for the  mixed  edges incident to a vertex repeatedly one by one.

\begin{corollary}\label{dv}
Let $D$ be a mixed graph and  $v\in V(D)$. Let $\mathcal{C}_v$ be  the set of  all real  mixed cycles   in $D$  containing $v$. Then
$$\mathnormal{\Phi}(D, \lambda)=\lambda\mathnormal{\Phi}(D-v, \lambda)-\sum_{uv\in E(D)}\mathnormal{\Phi}(D-u-v, \lambda)-2\sum_{C\in\mathcal{C}_v}h(C)\mathnormal{\Phi}(D-C, \lambda).$$
\end{corollary}

\section{The Hermitian spectral radii of mixed graphs}

In this section, we   present some results on Hermitian spectral radii of mixed graphs for   later use.
From  the Perron--Frobenius theorem, we know that the spectral radius of a connected  undirected graph strictly decreases by  deleting  a vertex or an edge from the   graph. However, the fact  does not hold for Hermitian spectral radius of a  connected  mixed graph.  We  give   a sufficient condition in the following theorem  generalizing  Theorem 3.2  in \cite{xu}.

\begin{theorem}\label{stt}
Let $D$ be a connected mixed graph all whose   real   mixed cycles    are    positive mixed even cycles. Then $\rho(D)>\rho(D-u)$ and $\rho(D)>\rho(D-e)$ for every  vertex  $u\in V(D)$ and edge   $e\in E(D)$.
\end{theorem}

\begin{proof}
We prove the assertion by induction on   $m=|E(D)|$.  The assertion clearly   holds for $m=1$.
Suppose   that the assertion  is valid for all connected mixed graphs of size less than $m$.  Consider a  connected mixed graph $D$ of size  $m$ and assume that  $e=uv$ is an arbitrary edge of $D$.  By Corollary  \ref{symmetric}, the Hermitian  spectrum of $D-e$ is symmetric about $0$,  so
$\rho=\rho(D-e)$ can be considered as the largest Hermitian  eigenvalue of $D-e$. We   first establish that $\mathnormal{\Phi}(D-u-v, \rho)>0$.
For this, we consider the following two cases.

\noindent{\bf Case 1.}   The edge   $e$ is a cut edge.

Denote the connected components of $D-e$ by $D_1$ and  $D_2$.
Assume  without loss of generality  that $u\in V(D_1)$ and $v\in V(D_2)$.
By the induction hypothesis, $\rho(D_1)>\rho(D_1-u)$ and
$\rho(D_2)>\rho(D_2-v)$. Therefore,
$$\rho(D-u-v)=\max\{\rho(D_1-u), \rho(D_2-v)\}
<\max\{\rho(D_1), \rho(D_2)\}
=\rho(D-e)=\rho.$$
This means that  $\mathnormal{\Phi}(D-u-v, \rho)> 0$.

\noindent{\bf Case 2.}  The edge   $e$ is not a cut edge.

By the induction hypothesis and the interlacing theorem,
$\rho=\rho(D-e)>\rho(D-u)\geqslant\rho(D-u-v)$ which means that  $\mathnormal{\Phi}(D-u-v, \rho)>0$.

Let $\mathcal{C}_e$ be  the set of  all real  mixed cycles   in $D$  containing $e$.
For any     $C\in\mathcal{C}_e$,  $D-C$ is  an induced subgraph of $D-u-v$ and so
by the   interlacing theorem, $\rho(D-C)\leqslant\rho(D-u-v)<\rho$  which yields that   $\mathnormal{\Phi}(D-C, \rho)>0$.
Since all   real mixed cycles of $D$ are    positive,  $h(C)=1$ for each $C\in\mathcal{C}_e$.  By   Theorem \ref{de},
$$\mathnormal{\Phi}(D, \rho)=\mathnormal{\Phi}(D-e, \rho)-\mathnormal{\Phi}(D-u-v, \rho)-2\sum_{C\in\mathcal{C}_e}h(C)\mathnormal{\Phi}(D-C, \rho)<0,$$
proving  that $\rho(D)>\rho=\rho(D-e)$.  Applying    the interlacing theorem,  $\rho(D)>\rho(D-e)\geqslant\rho(D-u)$.
\end{proof}

\begin{remark} The conditions in Theorem \ref{stt} can not be omitted.  Consider two mixed graphs $D_1$ and $D_2$ whose labeling are shown in Fig. 1. Notice that $D_1$ contains a real mixed odd cycle $v_2v_3v_4v_2$ and a negative mixed even cycle $v_1v_2v_3v_4v_1$, and $D_2$ contains a negative mixed even cycle $u_1u_2u_3u_4u_1$, no real mixed odd cycle.
By an easy calculation, it follows that $\rho(D_1)=\rho(D_1-v_3)=2$, $\rho(D_1-v_3v_4)\sim 2.170 > 2;$ $\rho(D_2)=\rho(D_1-u_2)=\sqrt{3}$, $\rho(D_2-u_1u_3)=2 > \sqrt{3}$.

\end{remark}

\begin{figure*}[h]
\psset{unit=1}
\begin{center}
\begin{pspicture}(0,8)(9.5,12.5)

\cnode(6,9){1.65pt}{21}
\cnode(9,9){1.65pt}{22}
\cnode(6,12){1.65pt}{23}
\cnode(9,12){1.65pt}{24}

\ncline{->}{21}{23}\ncline{-}{23}{24}\ncline{->}{22}{21}\ncline{->}{23}{22}\ncline{22}{24} \ncline{21}{24}

\rput(5.75,12.25){$u_1$} \rput(9.25,12.25){$u_2$} \rput(9.25,8.75){$u_3$} \rput(5.75,8.75){$u_4$}
\rput(7.5,8.2){$D_2$}

\cnode(0.5,12){1.65pt}{51}
\cnode(3.5,12){1.65pt}{52}
\cnode(3.5,9){1.65pt}{53}
\cnode(0.5,9){1.65pt}{54}

\ncline{->}{52}{53}\ncline{->}{53}{54}\ncline{51}{52}\ncline{54}{52}\ncline{51}{54}

\rput(0.25,12.25){$v_1$} \rput(3.75,12.25){$v_2$} \rput(3.75,8.75){$v_3$} \rput(0.25,8.75){$v_4$}
\rput(2,8.2){$D_1$}
\end{pspicture}
\caption{\footnotesize Mixed graphs $D_1$ and $D_2$. \label{aaa}}
\end{center}
\end{figure*}
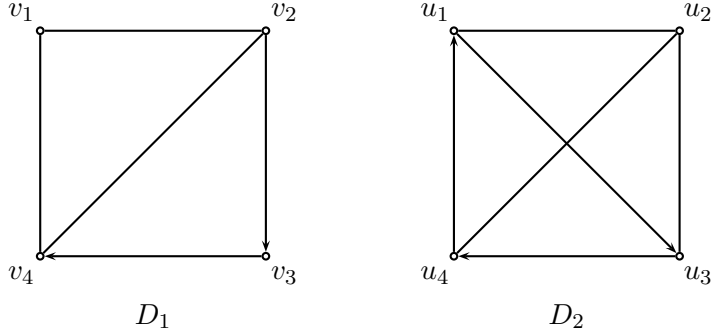

\begin{lemma}{\em \cite{mohar3}} \label{pen}
Suppose that a mixed graph $M$ is obtained from a connected mixed graph $N$ by attaching  a new vertex  to  a  vertex $u$  in $N$. If $x$ is an eigenvector of $N$ whose eigenvalue $\lambda$ satisfies
 $|\lambda|=\rho(N)$ and $x_u\neq0$, then $\rho(M)>\rho(N)$.
\end{lemma}

\begin{lemma}\label{cycle2}
Let $D$ be a connected unicyclic   mixed graph containing a  mixed cycle $C$. If  $\rho(D)=\rho(C)=2$, then $D=C$.
\end{lemma}

\begin{proof}
Let $C=C_n$. Applying  Lemma \ref{pen}, it is sufficient to show that any   eigenvector $x$  corresponding to  an eigenvalue $\lambda$ of $C$ with  $|\lambda|=\rho(C)=2$ has no   zero   components. If  $x_u=0$ for some $u\in V(C)$, then the vector obtained   from  $x$ by deleting   the $u$th component of $x$  is an   eigenvector   of  $C-u$ corresponding to $\lambda$. But, this is   impossible, since    the Hermitian  spectral radius of $C-u$   is less  than $2$ in view  of  Corollary  \ref{positive}. Towards a contradiction, suppose that  $D\neq C$.
Since $D$ is a connected unicyclic mixed graph,    there    exists a vertex $v\in V(D-C)$ adjacent to exactly one vertex on $C$. If   $H$ is  the induced  subgraph of $D$ on  $V(C)\cup\{v\}$,  then  Lemma \ref{pen} implies that  $\rho(H)>\rho(C)=2$, which  contradicts  $\rho(H)\leqslant\rho(D)=2$.
\end{proof}

\section{$C_4$-free mixed graphs whose spectral radii do not exceed 2}

In this section, we determine all $C_4$-free mixed graphs whose spectral radii do not exceed $2$. We first introduce some families of graphs to use later.
Let $P_n$ and $C_n$ denote respectively  the path and the cycle  on $n$ vertices.
A {\sl star-like} tree $S(n_1, \ldots, n_k)$ is an  undirected     tree with   a vertex $v$    such  that  $S(n_1, \ldots, n_k)-v=P_{n_1}\cup\cdots\cup P_{n_k}$.
Denote by  $Y(r, s, t)$   the tree consisting of the  path  $P_{r+s+t-1}$ whose vertices are ordered as  $v_1, \ldots, v_{r+s+t-1}$ with two extra pendant edges affixed at $v_r$ and $v_{r+s}$.
A {\sl dumbbell} graph, denoted by $D(r, s, t)$, is an  undirected  graph consisting of two vertex disjoint cycles $C_r$, $C_s$,  and a path    $P_t$ joining the cycles  having only its endpoints  in common with them.
A {\sl  theta}  graph, denoted by $\theta(r, s, t)$, is an  undirected  graph consisting  of three internally disjoint paths    $P_r, P_s, P_t$ with the same endpoints.

\begin{definition}
Consider the cycle  $C_n$ as $v_1v_2\cdots v_nv_1$.
Denote by $C_n(k_1, \ldots, k_n)$ the  undirected    graph obtained from $C_n$ by identifying   $v_i$ with   a pendent
vertex of $P_{k_i+1}$   for $i=1, \ldots, n$. We write $C_n(k_1, \ldots, k_t)$ instead of   $C_n(k_1, \ldots, k_t, 0, \ldots, 0)$ for simplicity  whenever
$k_{t+1}=\cdots=k_n=0$.
\end{definition}

The following theorem characterizes  the undirected graphs whose spectral radii   do  not exceed $2$.

\begin{theorem}{\rm \cite{Smith}} \label{undirected}
All undirected graphs whose spectral radii do not exceed $2$ are isomorphic to one of  the following undirected graphs or   their   subgraphs.
\begin{itemize}
\item[{\rm (i)}] $C_n$ for any integer $n\geqslant3$;
\item[{\rm (ii)}]  $Y(2, n-5, 2)$  for any integer   $n\geqslant5$;
\item[{\rm (iii)}] $S(1, 2, 5)$, $S(1, 3, 3)$,   $S(2, 2, 2)$.
\end{itemize}
\end{theorem}

\begin{definition}
Let $G$ be an undirected  graph. Denote by  $G^+$ (respectively, $G^-$, $G^*$)    the family of  mixed graphs with $G$ as their underlying graph, whose all  mixed cycles  are  positive (respectively,  negative, imaginary).
Denote by  $G^{\widehat{+}}$ (respectively, $G^{\widehat{-}}$, $G^{\widehat{*}}$)    the family of    mixed graphs  contained  in $G^+$ (respectively, $G^-$, $G^*$) and their induced  mixed subgraphs.
\end{definition}

\begin{remark} \label{all}
Let $G$ be an  undirected  graph. Using   Theorem \ref{cha},  it is  easy to see that
all mixed graphs in    $G^+$ ($G^-$, $G^*$) have the same Hermitian spectrum.
\end{remark}

The following consequence is obtained from  Corollary  \ref{undercycle},   Lemma \ref{cycle2} and the interlacing theorem.

\begin{corollary}\label{+-}
Any graph   in one of the following families  has  the Hermitian spectral radius  greater than $2$.
\begin{itemize}
\item[{\rm (i)}]  $C_n(1)^+$  for any integer  $n\geqslant3$;
\item[{\rm (ii)}]   $C_n(1)^-$  for any odd  number  $n\geqslant3$.
\end{itemize}
\end{corollary}

The {\sl girth}   of a mixed graph $D$    is the minimum  length of   cycles
in $G(D)$. The following theorem generalizes  the analogue  result   for oriented graphs  appeared  in \cite{gong}.

\begin{lemma} \label{containcycle}
Let $D$ be a connected $C_4$-free mixed graph with  $\rho(D)\leqslant2$. If $D$ is neither  a mixed tree nor a mixed cycle, then  $D$
is isomorphic to a mixed graph contained in  one of the  following families.
\begin{itemize}
\item[{\rm (i)}]  $C_3(2)^{\widehat{*}}$,  $C_6(1, 0, 1, 0, 1)^{\widehat{-}}$, $C_6(2, 0, 0, 2)^{\widehat{-}}$, $C_8(1, 0, 0, 0, 1)^{\widehat{-}}$;
\item[{\rm (ii)}] The family of  mixed graphs with   underlying graph    $\theta(3, 5, 5)$    containing    two negative mixed  cycles  $C_6$,  and their induced    mixed subgraphs.
\end{itemize}
\end{lemma}

\begin{proof}
Let $m$ be  the girth of $D$ and let $C$ be a mixed  cycle of length $m$ in $D$ as  $u_1u_2\ldots u_mu_1$.  We identify $G(C)$ with  $C_m$. If $m\geqslant5$, then   $G(D)$ contains an induced  subgraph isomorphic to  $C_m(1)$, since $D$ is connected and is not a mixed  cycle.
If $m\geqslant9$, then   $C_m(1)-u_5$   contains   $S(1, 3, 4)$ as an induced subgraph. This is a contradiction, since    $\rho(S(1, 3, 4))>2$   by     Theorem  \ref{undirected},   and $\rho(S(1, 3, 4))\leqslant\rho(D)\leqslant2$  by   Corollary \ref{positive} and  the interlacing theorem. If $m\in\{5, 7\}$, then   it follows from    Corollary  \ref{+-}  that   $C$ is imaginary.  By the Remark in Section 4 and an easy calculation, it is easy to show that the Hermitian spectral radius of any mixed graph  in    $C_5(1)^*\cup C_7(1)^*$    is greater than $2$, a contradiction. Therefore, $m\in\{3, 6, 8\}$.

\noindent {\bf Case 1.}  $m=3$.

Since $D$ is a  connected   $C_4$-free mixed graph which is not a mixed cycle, any triangle in  $G(D)$ is contained   in an induced  subgraph isomorphic to $C_3(1)$.
It follows from    Corollary  \ref{+-}  that  $C$ and  all other triangles in $D$ are  imaginary. Towards a  contradiction, suppose that two vertices on $C$ have neighbors in $V(D-C)$.   Since $D$ is $C_4$-free, $D$ contains an element of    $C_3(1, 1)^*$ as an induced subgraph.    But,  by the Remark in Section 4 and an easy calculation, it is easy to show that the Hermitian spectral radius of any mixed graph  in     $C_3(1, 1)^*$   is greater than $2$, a contradiction.
Without loss of generality, assume that   $u_1$ is the unique vertex on $C$ having a   neighbor outside $C$, say  $w_1$.
Noting that  all   triangles in $D$ are  imaginary and
using $\rho(K_{1, 4})=2$,   we conclude from  Theorem \ref{stt}  that  $d(u_1)=3$.
We may assume that $w_1$ has a neighbor other than  $u_1$, since otherwise, $D\in C_3(2)^{\widehat{*}}$, we are done. Since $D$ is $C_4$-free,   $N(w_1)\cap(N(u_2)\cup N(u_3))=\varnothing$.
Again, By the Remark in Section 4 and a routine calculation, it follows that the Hermitian spectral radius of any mixed graph in $C_3(2)^*$ is equal to $2$. Using this and noting that $D$ is $C_4$-free, we conclude from Theorem \ref{stt} that      $d(w_1)=2$.  Let  $w_2$ be the  neighbor of   $w_1$ other than $u_1$. Now, if $w_2$ is adjacent to a vertex other than  $w_1$, then $D$ contains a member   $H\in C_3(3)^*$ as an induced  subgraph. This contradicts Theorem \ref{stt},  since the Hermitian spectral radius of any mixed graph  in $C_3(2)^*$ is equal to $2$. Therefore,  $D\in C_3(2)^{\widehat{*}}$.

\noindent {\bf Case 2.}  $m=6$.

As we mentioned in the first paragraph of the proof, $G(D)$ contains  $C_m(1)$ as an induced subgraph.
The Remark and a routine  calculation show   that the Hermitian spectral radius of any mixed graph  in     $C_6(1)^*$   is greater than $2$. This along with     Corollary  \ref{+-} forces  that  $C$ is   negative.
By Theorem \ref{undirected}, $\rho(Y(3, 0, 3))>2$ and so it follows from the  interlacing theorem that any vertex on $C$ has  at most $3$ neighbors  in $D$.

\noindent {\bf Case 2.1.}    There is no mixed  cycle $C'\neq C$ in $D$ with $E(C)\cap E(C')\neq\varnothing$.

We first claim   that $D\in C_6(k_1, \ldots, k_6)^-$  for some  $k_1, \ldots, k_6$. Towards a contradiction and without  loss of generality, suppose  that the subgraph $T$    attached to $u_1$ is not a path.
So, the induced mixed subgraph of $D$ on  $V(C-u_3)\cup V(T)$    contains an  induced mixed subgraph   with the underlying graph     $Y(4, s, 2)$     for some $s\geqslant1$. This contradicts Theorem \ref{undirected}, proving the claim.
Now, since   $C_6(3)-u_4=S(3, 2, 2)$,  $C_6(1, 1)-u_4= Y(3, 1, 2)$,  and
$C_6(2, 0, 1)-u_5=Y(3, 2, 2)$,      it follows from Theorem \ref{undirected} and the interlacing theorem  that
$D$ contains no induced  subgraphs in $C_6(3)^-\cup C_6(1, 1)^-\cup C_6(2, 0, 1)^-$.
Therefore,   $D\in C_6(1, 0, 1, 0, 1)^-\cup C_6(2, 0, 0, 2)^-$.
It is  routine to verify that any mixed graph in $C_6(1, 0, 1, 0, 1)^-\cup C_6(2, 0, 0, 2)^-$ has the Hermitian  spectral radius $2$, we are done.

\noindent {\bf Case 2.2.}    There is a  mixed  cycle $C'\neq C$ in $D$ with $E(C)\cap E(C')\neq\varnothing$.

As we mentioned in Case 2.1, $D$ has no induced subgraphs in  $C_6(3)^-$. Further, an easy  calculation shows   that the Hermitian spectral radius of any mixed graph in    $C_7(1)^*$    is greater than $2$.
Using  these  facts and Corollary \ref{+-}, one deduces that $D$   contains no  mixed graph  in  $C_6(3)^-\cup C_7(1)^+\cup C_7(1)^-\cup C_7(1)^*$ as an induced mixed subgraph.   This implies that
the length of $C'$ must be $6$. Moreover, if  $|E(C)\cap E(C')|=1$, then $D$
contains an induced   mixed subgraph in  $C_6(3)^-$, a contradiction.  Consequently, the underlying graph of the induced mixed subgraph  $H$ of
$D$  on $V(C)\cup V(C')$ is either     $\theta(5, 5, 3)$   or $\theta(4, 4, 4)$. On the other hand, since $C$ and $C'$ are negative, then  the   third mixed  cycle in the induced subgraph  of $D$  on $V(C)\cup V(C')$ must be  positive.
Using Corollary \ref{+-},   any mixed graph in  $C_6(1)^+$ has  the Hermitian  spectral radius greater than  $2$. This yields  that $G(H)=\theta(3, 5, 5)$. Without loss of generality, assume that $V(C)\cap V(C')=\{u_1, u_2, u_3\}$.
As we mentioned in Case 2.1,  $D$ has no induced subgraphs in  $C_6(1, 1)^-$. Since the girth of $D$ is $6$, one  concludes  that the degree of $u_2$ in $D$ must be $2$.
Furthermore, Corollary \ref{+-}  implies   that  $D$ contains no  mixed graph in $C_8(1)^+$. This along with   the connectivity of $D$ forces  that $D=H$, as required.

\noindent {\bf Case 3.} $m=8$.

By the Remark in Section 4 and an easy  calculation, it is easy to show that the Hermitian spectral radius of any mixed graph in $C_8(1)^*$ is greater than $2$.
This along with Corollary  \ref{+-} forces  that $C$ is negative.
Since  $C_8(2)-u_5=S(2, 3, 3)$,   $C_8(1, 1)-u_4=Y(2, 1, 5)$,  $C_8(1, 0, 1)-u_4=C_8(1, 0, 0, 1)-u_5=S(1, 3, 4)$,
Theorem \ref{undirected} along with the interlacing theorem  forces that $D\in C_8(1, 0, 0, 0, 1)^-$.  Note that the Hermitian spectral radius of any mixed graph in    $C_8(1, 0, 0, 0, 1)^-$    is equal to $2$ by an easy calculation. The result follows.
\end{proof}

Now we are in the position to state our main theorem which is obtained by Theorem \ref{undirected} and Lemma \ref{containcycle}.

\begin{theorem}\label{1main}
Let $D$ be a connected $C_4$-free mixed graph with  $\rho(D)\leqslant2$.  Then  $D$
is a mixed graph contained in  one of the  following families.
\begin{itemize}
\item[{\rm (i)}]  All mixed graphs with one of the undirected graphs $C_n$, $Y(2, n-5, 2)(n \geq 5)$,
$S(1, 2, 5)$, $S(1, 3, 3)$, $S(2, 2, 2)$ as their underlying graphs,  and their induced mixed subgraphs;
\item[{\rm (ii)}]  $C_3(2)^{\widehat{*}}$,  $C_6(1, 0, 1, 0, 1)^{\widehat{-}}$, $C_6(2, 0, 0, 2)^{\widehat{-}}$, $C_8(1, 0, 0, 0, 1)^{\widehat{-}}$;
\item[{\rm (iii)}] The family of  mixed graphs with   underlying graph    $\theta(3, 5, 5)$    containing    two negative  mixed cycles  $C_6$,   and their   induced  mixed subgraphs.
\end{itemize}
\end{theorem}

By a routine calculation and checking the proof of Lemma  \ref{containcycle}, we  get the following corollary as the end of the paper.

\begin{corollary}
Let $D$ be a $C_4$-free mixed graph with $\rho(D)=2$. Then $D$ is  a mixed graph contained in  one of the  following families.
\begin{itemize}
\item[{\rm (i)}]   All mixed graphs with one of the undirected graphs $Y(2, n-5, 2)(n \geq 5)$,
$S(1, 2, 5)$, $S(1, 3, 3)$, $S(2, 2, 2)$ as their underlying graphs;
\item[{\rm (ii)}]   $C_n^+$  for any integer $n\geqslant3$ and $C_n^-$ for any odd number  $n\geqslant3$;
\item[{\rm (iii)}]  $C_3(2)^*$,   $C_6(1, 0, 1)^-$,   $C_6(1, 0, 1, 0, 1)^-$, $C_6(2)^-$, $C_6(2, 0, 0, 1)^-$,  $C_6(2, 0, 0, 2)^-$, $C_8(1)^-$,    $C_8(1, 0, 0, 0, 1)^-$;
\item[{\rm (iv)}] The family of  mixed graphs with   underlying graph    $\theta(3, 5, 5)$    containing    two negative  mixed cycles  $C_6$.
\end{itemize}
\end{corollary}

\end{document}